\documentclass[11pt]{amsart}

\usepackage{bm}
\usepackage{fullpage}
\usepackage{amssymb}
\usepackage{amsmath,amsfonts,amsthm}
\usepackage{hyperref}
\usepackage{color}
\usepackage{cite}

\newtheorem{alphtheorem}{Theorem}

\newtheorem{alphlemma}{Lemma}

\newtheorem{theorem}{Theorem}

\newtheorem{lemma}[theorem]{Lemma}

\newtheorem{claim}[theorem]{Claim}

\newtheorem{corollary}[theorem]{Corollary}

\DeclareMathOperator\G{\mathcal{G}}

\title{ A Catlin-type Theorem for Graph partitioning Avoiding prescribed  Subgraphs}

\author{Yaser Rowshan and Ali Taherkhani}
\address{Y. Rowshan, 
Department of Mathematics, Institute for Advanced Studies in Basic Sciences (IASBS), Zanjan 45137-66731, Iran}
\email{y.rowshan@iasbs.ac.ir}
\address{A. Taherkhani, 
Department of Mathematics, Institute for Advanced Studies in Basic Sciences (IASBS), Zanjan 45137-66731, Iran}
\email{ali.taherkhani@iasbs.ac.ir}
\begin{document}
\maketitle 

\begin{abstract} 
As an extension of the Brooks theorem, Catlin in 1979 showed 
that if $H$ is neither an odd cycle nor a complete graph with maximum degree $\Delta(H)$, then $H$ has a  vertex $\Delta(H)$-coloring such that one of the color classes is a maximum independent set. 
 Let  $G$ be a connected graph of order at least $2$. A $G$-free $k$-coloring of a graph $H$ is a partition of the vertex set of $H$ into  $V_1,\ldots,V_k$ such that $H[V_i]$, the  subgraph induced on $V_i$, does not contain 
  any subgraph isomorphic to $G$.
  As a generalization of Catlin's theorem we show that a graph
   $H$ has a $G$-free $\lceil{\Delta(H)\over \delta(G)}\rceil$-coloring for which 
   one of the color classes is a maximum $G$-free subset of $V(H)$ if 
   $H$ satisfies the following conditions; (1) $H$ is  not isomorphic to $G$
  if $G$ is regular, (2) 
 $H$ is not isomorphic to $K_{k\delta(G)+1}$ if $G \simeq K_{\delta(G)+1}$, and (3)
 $H$ is not an odd cycle  
 if $G$ is isomorphic to $K_2$.
   Indeed,
   we show even more, by proving that if $G_1,\ldots,G_k$
 are  connected graphs with minimum degrees  $d_1,\ldots,d_k$, respectively, and $\Delta(H)=\sum_{i=1}^{k}d_k$,
 then there is a partition of vertices of $H$ to $V_1,\ldots,V_k$
such that each $H[V_i]$ is $G_i$-free and moreover one of  $V_i$s  can be chosen in a way that $H[V_i]$  is a maximum   $G_i$-free subset of $V(H)$ except  either $k=1$ and $H$ is  isomorphic to  $G_1$, each
$G_i$ is isomorphic to $K_{d_i+1}$ and $H$ is not isomorphic to $K_{\Delta(H)+1}$, or
     each $G_i$ is isomorphic to $K_{2}$ and $H$ is not an odd cycle.


\end{abstract}

\section{Introduction} 
In this paper, we are only concerned with simple graphs and we follow \cite{bondy} for terminology and notations not defined here. For a given graph $G$, we denote its vertex set, edge set, maximum degree, and minimum degree by $V(G)$, $E(G)$, $\Delta(G)$, and $\delta(G)$, respectively. For a vertex $v\in V(G)$, we use $\deg_G{(v)}$ (or simply $\deg{(v)}$) and $N_G(u)$ to denote the degree and the set of neighbors of $v$ in $G$, respectively. The subgraph of $G$ induced on $X \subseteq V(G)$ is denoted by $G[X]$. 

A  $k$-coloring of $G$ is an assignment of $k$ colors to its vertices so that no two adjacent vertices receive the same color.
The chromatic number of $G$,  denoted by $\chi(G)$,
is the minimum number $k$ for which  G has a $k$-coloring.
 It is known that for any graph $G$, we have $\chi(G)\leq \Delta(G)+1$. Brooks showed that  if a connected graph $G$  is neither an odd cycle nor a complete graph, then $\chi(G)\leq \Delta(G)$~\cite{Brooks}.

The conditional chromatic number $\chi(H,P)$ of $H$, with respect to a graphical property $P$, is the minimum number $k$ for which there 
is a partition of  $V(H)$  into  sets $V_1,\ldots, V_k$
 such that for each $1\leq i\leq k$, the induced subgraph $H[V_i]$ satisfies the property $P$.
 This generalization of graph coloring was introduced by Harary in 1985~\cite{MR778402}. In this sense, for an ordinary graph coloring, the subgraph  induced on each $V_i$ of the partition does not contain $K_2$.  As another special case, when $P$ is the property of being acyclic, $\chi(H,P)$ is called the vertex arboricity of $H$. In other words, the vertex arboricity of a graph $H$, denoted by $a(H)$, is the minimum number $k$ for which $V(H)$ can be decomposed into subsets $V_1,\ldots, V_k$ so that each subset induces an acyclic subgraph.  The vertex arboricity of graphs was first introduced by Chartrand, Kronk, and Wall in~\cite{Chartrand}. 
Also, it has been shown that for any arbitrary graph, say $H$, $a(H)\leq \lceil \frac{\Delta(H)+1}{2}\rceil$~\cite{Chartrand}, while a Brooks-type theorem was proved in~\cite{Kronk}. If $H$ is not a cycle or a complete graph of odd order, then we have $a(H)\leq \lceil \frac{\Delta}{2}\rceil$\cite{Kronk} and for a planar graph $H$, it has  been shown that  $a(H)\leq 3$ \cite{Chartrand,Hedet}. Moreover, for $k\in\{3,4,5,6\}$, and every planar graph $H$ with no subgraph isomorphism to $C_k$, we have $a(H)\leq 2$~\cite{ras} (for more results on arboricity  see e.g.\cite{Catlin,Catlin1,Chartrand,Chartrand1,Borodin,Bauer,Harary,Hedet,Kronk, ras}).

When $P$ is the property of {\it not containing a subgraph isomorphic to $G$}, we write $\chi_{_G}(H)$ instead of $\chi(G,P)$ which is called the $G$-free chromatic number, henceforth. In this regard, we say a graph $H$ has a $G$-free $k$-coloring if there is a map $c : V(H) \longrightarrow \{1,2,\ldots,k\}$ such that the subgraph induced on each one of the  color classes of $c$ is $G$-free. One can see that an ordinary  $k$-coloring  is a $K_2$-free coloring of a graph $H$  with $k$ colors. Also, for any graph $H$, one may show that 
$\chi_{_G}(H)\leq \lceil \frac{\chi(H)}{\chi(G)-1} \rceil.$

 In 1941 Brooks proved that  for a connected graph $H$, $\chi(H)\leq \Delta(H)$ when $H$ is neither an odd cycle nor a complete graph. 
 As an extension of Brooks' theorem, Catlin showed that 
  if $H$ is neither an odd cycle nor a complete graph, then $H$ has a proper $\Delta(H)$-coloring for which one of the color classes is a maximum independent set of $H$~\cite{Catlin}. Here, we prove an extension of Catlin's result for  partitioning  of the vertex set of a graph $H$ in a way that each class avoids having a  prescribed  subgraph. Clearly, in this way, we obtain a Brooks-Catlin-type theorem for the $G$-free chromatic number of a graph $H$ as follows.

 \begin{theorem}\label{1th}
  Let $k\geq 1$ be a positive integer. Assume that $G_1,\ldots,G_k$
 be  connected graphs with minimum degrees  $d_1,\ldots,d_k$, respectively, and $H$ be a connected graph with maximum degree
   $\Delta(H)$ where $\Delta(H)=\sum_{i=1}^{k}d_k$. Assume that $G_1,G_2,\ldots,G_k$,
   and $H$ satisfy the following conditions;
 \begin{itemize}
 \item If $k=1$, then $H$ is not isomorphic to  $G_1$.
 \item If  $G_i$ is isomorphic to $K_{d_i+1}$ for each $1\leq i\leq k$, then $H$ is not isomorphic to $K_{\Delta(H)+1}$.
  \item If  $G_i$ is isomorphic to $K_{2}$ for each $1\leq i\leq k$, then $H$ is neither an odd cycle nor a complete graph.
  \end{itemize}
Then, there is a partition of vertices of $H$ to $V_1,\ldots,V_k$
such that each $H[V_i]$ is $G_i$-free and moreover one of  $V_i$s  can be chosen in a way that $H[V_i]$  is a maximum induced  $G_i$-free subgraph in $H$. 
\end{theorem}
 In Theorem~\ref{1th}, if we take 
 $G_i=K_2$ for $1\leq i\leq k$, then we get Catlin's result. Also, if for a given graph $G$ and for $1\leq i\leq k$  we choose  $G_i=G$, we obtain the following Brooks-Catlin-type result for $G$-free coloring of graphs.
 \begin{corollary}\label{2th}
  Let $G$ be a connected graph with minimum degree $\delta(G)\geq 1$. Also, assume that $H$ is a connected graph with maximum degree
   $\Delta(H)$ while $H$ satisfies the following conditions;
 \begin{itemize}
 \item If $G$ is regular, then $H\ncong G$.
 \item If $G$ is isomorphic to $K_{\delta(G)+1}$, then $H$ is not $K_{k\delta(G)+1}$.
 \item If $G$ is isomorphic to $K_2$, then $H$ is neither an odd cycle nor a complete graph.
  \end{itemize}
Then, there is a $G$-free $\lceil \frac{\Delta(H)}{\delta(G)}\rceil$-coloring of $H$ such that one of whose color classes is a maximum induced  $G$-free subgraph in $H$. In particular,
\[\chi_{_G}(H)\leq\lceil \frac{\Delta(H)}{\delta(G)}\rceil.\]
\end{corollary}

An analogue to Catlin's result for vertex arboricity is due to Catlin and Lai~\cite{Catlin1}.
They proved the following interesting theorem for the vertex arboricity of graphs.
 \begin{alphtheorem}\label{Catlin-Lai}{\rm \cite{Catlin1}}
Assume that $H$ is neither a cycle nor a complete graph of odd order. 
\begin{itemize}
\item If $\Delta(H)$ is even, then  there is a coloring with $ \frac{\Delta(H)}{2}$ colors such that
each color class induces an acyclic subgraph and one of those  is a maximum induced acyclic subgraph in $H$.
\item If $\Delta(H)$ is odd, then  there is a coloring with $ {\lceil \frac{\Delta(H)}{2}\rceil}$ colors such that
each color class induces an acyclic subgraph. Moreover, this coloring
 can be chosen to satisfy one  of the following properties:
\begin{itemize}
\item[(a)] one  color class is an independent set and one  color class  is a maximum induced acyclic subgraph in $H$. 
\item[(b)] one  color class is a maximum independent set in $H$. 
 \end{itemize}
 \end{itemize}
 \end{alphtheorem}
 
 Let $\G$ be a family of graphs. For a graph $H$, a subset $W$  of   $V(H)$ is said to be $\G$-free if $H[W]$ does 
 not contain any one of the  members of $\G$. Therefore, a $\G$-free coloring of graph may be defined similarly. For example, if the family $\G$ consists of all connected graphs with minimum degree at least $2$, 
 then the $\G$-free chromatic number of a graph $H$ is equal to  the vertex arboricity of $H$.
  We define the minimum degree of $\G$ by  $\delta(\G)=\min\{\delta(G)| G\in\G\}$. 
  In this setup, it is straight forward to a generalization of Theorem~\ref{1th} as follows.
   \begin{theorem}\label{sec-thm}
  Let $k\geq 1$ be a positive integer. Assume that $\G_1,\ldots,\G_k$
 be $k$ families of connected graphs with minimum degrees  $d_1,\ldots,d_k$, respectively. Also, assume that $H$ is a connected graph with maximum degree
   $\Delta(H)$ where $\Delta(H)=\sum_{i=1}^{k}d_k$.
   Let $\G_1,\G_2,\ldots,\G_k$,
   and $H$ satisfy the following conditions;
 \begin{itemize}
 \item If $k=1$, then $H\not\in \G_1$.
 \item If  $K_{d_i+1}\in \G_i$ for each $1\leq i\leq k$, then $H$ is not isomorphic to $K_{\Delta(H)+1}$.
  \item If  $K_{2}\in \G_i$ for each $1\leq i\leq k$, then $H$ is neither an odd cycle nor a complete graph.
  \end{itemize}
Then, there is a partition of vertices of $H$ to $V_1,\ldots,V_k$
such that each $H[V_i]$ is $\G_i$-free and moreover one of  $V_i$s  can be chosen in a way that $H[V_i]$  is a maximum induced  $\G_i$-free subgraph in $H$. 
\end{theorem}

A graph $H$ is said to be $p$-degenerate if every subgraph of $H$ has a vertex of degree at most $p$.
 Let the  family $\G^{^{\geq p}}$ consist of all connected graphs with minimum degree at least $p$. 
 One may show that  being $p$-degenerate  is equivalent to not  containing any subgraph isomorphic to any one of the members of $\G^{^{\geq p+1}}$. Therefore, the Catlin-Lai theorem and  the next theorem due to Matmala are direct consequences of Theorem~\ref{sec-thm}.
\begin{alphtheorem}\label{thm:mat}{\rm\cite{MR2327961}}
Let $H$ be a graph with maximum degree $\Delta(H)\geq 3$ and $\omega(H)\leq \Delta(H)$. If $\Delta(H)= d_1+d_2$, then the vertices of $H$ can be partitioned into two sets $V_1$ and $V_2$ such that  $H[V_1]$ is a maximum $(d_1-1)$-degenerate induced subgraph and $H[V_2]$ is $(d_2-1)$-degenerate.  
\end{alphtheorem}

One can easily show that the following result due to Bollob{\'a}s and Manvel can be extended to their $\G_i$-free versions (instead of ($d_i-1$)-degeneracy). 
\begin{alphlemma}\label{lem1}{\rm\cite{Bollob}}
Let $H$ be a graph with maximum degree $\Delta(H)\geq 3$ and $\omega(H)\leq \Delta(H)$.
 If $\Delta(H)= d_1+d_2$, then the vertices of $H$ can be partitioned into two sets $V_1$ and $V_2$ such that $\Delta(H[V_1])\leq d_1$, $\Delta(H[V_2])\leq d_2$, $H[V_1]$ is $(d_1-1)$-degenerate and $H[V_2]$ is $(d_2-1)$-degenerate.  
\end{alphlemma}

Also, it is worth mentioning the following result  of Lov{\'a}sz which has a  close relation to the previous  result of Bollob{\'a}s and Manvel.
\begin{alphtheorem}\label{LV}{\rm\cite{Lovasz}} If $d_1\geq d_2\geq\ldots\geq d_k$ are positive integers  such that $d_1+d_2+\ldots+d_k \geq \Delta(H)+1$, then $V(H)$ can be decomposed into subsets $V_1,V_2,\ldots,V_k$, such that $\Delta(H[V_i])\leq d_i-1$ for each $1\leq i\leq k$.
\end{alphtheorem}
Note that if one chooses $k=\Delta(H)+1$ and $d_1=\cdots=d_k=1$, then this result implies that $\chi(H)\leq \Delta(H)+1$. Also, it is instructive to note that $\Delta(H)+1$ can not be replaced by $\Delta(H)$ in Theorem~\ref{LV}. To see this, consider the following example. Set $H=K_{3,3,3}$ which has maximum degree $6$ and assume that
$k=2$ and $d_1=d_2=3$, and note that there is not any decomposition of vertices of $K_{3,3,3}$ to subsets $V_1$ and $V_2$ such that 
$\Delta(H[V_i])\leq 2$. Of course, one can find some other nontrivial examples, too. Moreover,  one may construct a graph $H$ for which $\Delta(H)=d_1+d_2$ and
$H$ can not  be decomposed into two subsets $V_1,V_2$ such that $\Delta(H[V_i])\leq d_i-1$ for each $i\in\{1,2\}$ (see \cite{Bollob}).
Also, if $H$ and $G$ are connected graphs with maximum degrees $\Delta(H)$ and $\Delta(G)$, respectively, then, as a consequence of Theorem~\ref{LV} we have
$\chi_{_G}(H)\leq\lceil \frac{\Delta(H)+1}{\Delta(G)}\rceil \leq\lceil \frac{\Delta(H)+1}{\delta(G)}\rceil.$
\section{ Proofs}
\noindent The following lemma is the main part of  the proof of Theorem~\ref{1th}.
\begin{lemma}\label{le0} Let $G$ and $H$ be two connected graphs, where $G$ has the minimum degree $d$
and $H$ has the maximum degree $\Delta(H)$ where $\Delta(H)\geq d\geq 1$.
Assume that $S\subseteq V(H)$, $H[S]$ is  $G$-free and  $S$ has the maximum possible size.
 Suppose that $H\setminus S$ has as few  connected $(\Delta(H)-d)$-regular  subgraphs as possible. 
 Also, suppose that $H[S]$ has the minimum possible number of connected components.
 If $H\setminus S$ has a 
 $(\Delta(H)-d)$-regular connected subgraph, say $H_0$, then
\begin{itemize}
\item[(a)] for any vertex  $v\in V(H_0)$, $|N(v)\cap S|=d$,
\item[(b)]  the induced subgraph $H[S\cup\{v\}]$ has a unique copy of $G$, say $G_v$, 
such that $G_v$ is a $d$-regular component of $H[S\cup\{v\}]$, and
\item[(c)] Either $G$ is isomorphic to $ K_{d+1}$  and $H$ is isomorphic to $ K_{\Delta(H)+1}$,  
 $G=K_2$  and $H$ is isomorphic to $C_{2\ell+1}$ for some positive integer $\ell$,  
or $H$ is isomorphic to $ G$.
\end{itemize}
\end{lemma}
\begin{proof}
By the maximality of $S$, for each vertex  $v\in V(H)\setminus S$,  $H[S\cup\{v\}]$ has a copy of $G$.
Therefore, $ |N(v)\cap S|\geq d$
and consequently
$$\Delta(H\setminus S)\leq \Delta(H)-d.$$
Thus,  $H_0$ is a connected component of $H\setminus S$. Hence, for any $v\in V(H_0)$,
$|N(v)\cap (V(H)\setminus S)|= \Delta(H)-d$. Consequently, for any $v\in V(H_0)$, we have $ |N(v)\cap S|= d.$

To prove Part~(b), let $v\in V(H_0)$ and $G_v$  be a copy of $G$ in $H[S\cup\{v\}]$. Since $G$ has minimum   degree
$d$  and $|N(v)\cap S|=d$, we have $N(v)\cap S\subseteq V(G_v)$.

\begin{claim}\label{c6}
The subgraph $G_v$  is a unique copy of $G$ in $H[S\cup\{v\}]$ and moreover $G_v$ is a $d$-regular graph.
\end{claim}
\begin{proof}[Proof of Claim~\ref{c6}] By contradiction  suppose that there are  two copies of $G$ in 
 $H[S\cup\{v\}]$ such that these two  copies of $G$ in $ H[S\cup\{v\}]$ have different  vertex sets.
If  $u\in S\cup\{v\}$ and $d_{H[S\cup\{v\}]}(v,u)\leq 1$, then $u$ lies in all copies of $G$ in $H[S\cup\{v\}]$.
Now, let $i\geq 1$ be the largest positive integer such that for any vertex $u\in S\cup\{v\}$ with $d_{H[S\cup\{v\}]}(v,u)=i$, we have 
$u$ lies in all copies of $G$ in $ H[S\cup\{v\}]$.   Since there exist at least two  copies of $G$ with different vertex sets, there exists at least one vertex  $w\in S$ and a copy of $G$ in  $ H[S\cup\{v\}]$, say $G^*$, such that $d_{H[S\cup\{v\}]}(v,w)=i+1$ and $w\notin V(G^*)$.
Therefore, there is at least one neighbor of $w$ in $S$, say $y$, such that $d(v,y)= d(v,w)-1=i$. 
 Since  $d(v,y)=i$, $y$ lies in all copies of $G$ in $H[S\cup\{v\}]$. Set $S_1=(S\setminus\{y\})\cup\{v\}$.  Note that $|S_1|=|S|$ and $H[S_1]$ is  $G$-free because $y$ lies in all copies of $G$ in $ H[S\cup\{v\}]$. Since  $y$ is in at least two copies of $G$ in $H[S\cup\{v\}]$ and one of them does not contain $w$, we have
 $|N(y)\cap(S\cup\{v\})|=|N(y)\cap S_1|\geq d+1$. Therefore, $|N(y)\cap (V(H) \setminus S_1)|\leq \Delta(H)-d-1$. As a consequence $y$ does not lie in any $(\Delta(H)-d)$-regular subgraph in  $H\setminus S_1$. Hence, the number  of
  $(\Delta(H)-d)$-regular connected subgraphs of $H\setminus S_1$ is less than that of $H\setminus S$, 
 which contradicts the assumption that $H\setminus S$ has as few $(\Delta(H)-d)$-regular subgraphs as possible. 
 Thus,  $H[S\cup\{v\}]$ contains only one copy of $G$.
 
Now assume that all copies of $G$ in $ H[S\cup\{v\}]$ has the same vertex set. If there  exist at least two distinct copies of $G$ in $ H[S\cup\{v\}]$ with the same vertex set, then there is a vertex $u\in V(G_v)\subseteq S$ such that 
$|N(u)\cap (S\cup\{v\})|\geq d+1$. Define $S_1=(S\setminus\{u\})\cup\{v\}$. Since $H[S_1]$ is $G$-free, $|S_1|=|S|$, and 
$H\setminus S$ has as few $(\Delta(H)-d)$-regular connected subgraphs as possible, so $u$ must lie in a $(\Delta(H)-d)$-regular subgraph in $H\setminus S_1$.
Therefore, $|N(u)\cap (V(H)\setminus(S\cup\{v\}))|\geq \Delta(H)-d $ and consequently $|N(u)|=\deg(u)\geq \Delta(H)+1$, which is a contradiction.
 \end{proof}
 Assume that $G_v$ is a subgraph of
$H[S\cup\{v\}]$ but is not one of its connected components.
Thus, 
there is at least one vertex of $G_v$, say $u$, such that $|N(u)\cap (S\cup\{v\})|\geq d+1$.
Therefore, using the same reasoning as the previous paragraph we can prove that $G_v$ is a component of $H[S\cup\{v\}]$.
\begin{claim}\label{c7}
The subgraph $G_v$ is a component of $H[S\cup\{v\}]$.
\end{claim}
To prove Part~{\rm (c)}, set $S_0=S$. In view of  Part~(a), we have
 $H_0$ is a component of $H\setminus S_0$. 
 
 Assume that $H_0$ has only one vertex, say $v$. Now by Claims \ref{c6} and \ref{c7}, $H[S_0\cup\{v\}]$ has a unique copy of $G$, which is a component of $H[S_0\cup\{v\}]$. Since $H$ is connected and  $\Delta(H)=d$, we have $H$ is isomorphic to $G$.   Assume that $|V(H_0)|=2$ and $d=1$.
Then, $H_0\cong K_2$ and  from Claim~\ref{c6} we have $G\cong K_2$. 
 Consequently, $\Delta(H)=2$. Since $H$  is connected, we have $H$ is path or cycle. 
As $S$ independent set of maximum size and $H\setminus S$ has a copy of $K_2$,  $H$ must be an odd cycle.
  Therefore, we may assume that either $|V(H_0)|\geq 3$ or $d\geq 2$.

  Let $v\in V(H_0)$. By using Part~(b),
$H[S_0\cup \{v\}]$ has a unique copy of $G$, say $G_{v}$.
\begin{claim}\label{cc} Let $v$ be a vertex of $V(H_0)$ which is not a cut vertex.  If for some $w\neq v$ in  $V(H_0)$ we have 
$G_{v}\setminus\{v\}=G_{w}\setminus\{w\}$, then the statement of Part~{\rm (c)} holds.
\end{claim}
 \begin{proof}[proof of Claim \ref{cc}.]
Since $G_v$ and $G_w$ are $d$-regular and $G_{v}\setminus\{v\}=G_{w}\setminus\{w\}$, we have
$N(v)\cap S_0=N(w)\cap S_0$.
 We show that $H[N(v)\cap S_0]\cong K_{d}$.
 By contradiction assume that there  exist two vertices $y$ and $y'$ in $N(v)\cap S_0$ such that $yy'\notin E(H)$. Define $S^*=(S_0\setminus\{y,y'\})\cup \{v,w\}$. One can check that $|S^*|=|S|$ and $|N(v)\cap S^*|=|N(w)\cap S^*|\leq d-1$ and hence $H[S^*]$ is $G$-free. Since   $|N(y)\cap S^*|\geq d+1$ and
$|N(y')\cap S^*|\geq d+1$, $y$ and $y'$ are not in any  $(\Delta(H)-d)$-regular subgraph in $H\setminus S^* $. Hence, 
the number of $(\Delta(H)-d)$-regular connected  subgraphs of $H\setminus S^*$ is less than that  of  $H\setminus S_0$, which contradicts the assumption that
 $H\setminus S_0$ has as few $(\Delta(H)-d)$-regular connected subgraphs as possible. Hence, for every two vertices $y$ and $y'$ in $N(v)\cap S_0$, $yy'\in E(H)$.   Therefore, $H[N(v)\cap S_0]\cong K_{d}$ and moreover $G_{v}\cong K_{d+1}$. 
 
For every  vertex $y\in N(v)\cap S_0$,   we shall show that the subgraph induced by $N(y)\setminus S_0$ is isomorphic to 
the complete graph $K_{\Delta(H)-d+1}$.
Define  $S_1=(S_0\setminus\{y\})\cup\{v\}$. One can check that   $|S_1|=|S_0|$ and $H[S_1]$ is $G$-free. 
Since  $H\setminus S_0$ has as few  $(\Delta(H)-d)$-regular connected subgraphs  as possible,
 $y$ must lie in a   $(\Delta(H)-d)$-regular connected subgraph  in $H\setminus S_1$, say $H_1$. 
Therefore,   the number of neighbors of $y$  in $H\setminus S_0$ is $\Delta(H)-d+1$.

As $ N(w)\cap S_0=N(v)\cap S_0$, we have $y$ is adjacent to $w$ and moreover $w\in V(H_1)$. Since $H_0\setminus v$ is connected and $w\in V(H_0)\cap V(H_1)$, we have $(V(H_0)\setminus v)\subseteq V(H_1)$; otherwise there is a vertex in $V(H_0)\cap V(H_1)$
 has degree greater than $\Delta(H)$, which is not possible.  
Since $H_0$ and $H_1$ are  $(\Delta(H)-d)$-regular, so $N(v)\setminus S_0=N(y)\setminus (S_0\cup\{v\})$. Therefore,
$N(y)\setminus S_0$ is a subset of $V(H_0)$.
Assume that $v'$ and $v''$ are two neighbors of $y$  in $H\setminus S_0$. We show that $v'v''\in E(H)$.
On the contrary,  assume that $v'$ is not adjacent to $v''$. By Part~(a), $|N(v')\cap S_0|=|N(v'')\cap S_0|=d$.
Define  $S_2=(S_0\setminus\{y\})\cup\{v',v''\}$. One can check that   $|S_2|=|S_0|+1$ and 
$|N(v')\cap S_2|=|N(v'')\cap S_2|=d-1$. Hence,
 $H[S_2]$ is $G$-free, which  contradicts  the maximality of $S_0$. Therefore, $v'v''\in E(H)$ and consequently
 the subgraph induced by $N(y)\setminus S_0$ is isomorphic to the complete graph $K_{\Delta(H)-d+1}$. Therefore, 
 $H_0\cong K_{\Delta(H)-d+1}$.

For any two vertices  $y$ and $y'$ in $ N(v)\cap S_0$, we shall show $N(y)\setminus S_0=N(y')\setminus S_0$.
The vertex $v$ belongs to $(N(y)\cap N(y')) \setminus S_0$. On the contrary, suppose that there is a vertex $u\in N(y)\setminus S_0$
such that 
$u\not \in N(y')\setminus S_0$. As $v$ is adjacent to $u$ and $v\in  N(y')\setminus S_0$, so $v$  
has at least  $\Delta(H)-d+1$ neighbors in $H\setminus S_0$. Consequently, $\deg(v)\geq \Delta(H)+1$ which is not possible. Therefore, every vertex $y\in N(v)\cap S_0$ is adjacent to all vertices of $H_0$.
Since $H[N(v)\cap S_0]$ is isomorphic to $K_d$, the subgraph induced by 
$N(y)\setminus S_0$ is isomorphic to the complete graph $K_{\Delta(H)-d+1}$, and 
  every vertex $y\in N(v)\cap S_0$ is adjacent to all vertices of $H_0$, we conclude that
  $H[N(v)\cap S_0]\vee H_0\cong K_{\Delta(H)+1}$ is a  subgraph of $H$. Since $H$ is connected, we have 
$H\cong  K_{\Delta(H)+1}$.
\end{proof}

\noindent Now  assume that for  two vertices $v,w$  in $V(H_0)$ we have
 $G_{v}\setminus\{v\}\neq G_{w}\setminus\{w\}$.
 \begin{claim}\label{c8}
 If  $G_{v}\setminus\{v\}\neq G_{w}\setminus\{w\}$, then
$N(v)\cap N(w)\cap S_0=\varnothing$.
\end{claim}

 \begin{proof}[proof of Claim \ref{c8}.]
If $V(G_{v}\setminus\{v\})=V(G_{w}\setminus\{w\})$, then $E(G_{v}\setminus\{v\})\setminus E(G_{w}\setminus\{w\})\neq\varnothing$ and hence we can find a vertex in $G_{w}$ with degree greater than $d$. This is a contradiction because from Part~(b), $G_{w}$ is a $d$-regular component of $H[S_0\cup\{w\}]$. 
Therefore, there exists at least one vertex $u\in V(G_{v}\setminus\{v\})\setminus V(G_{w}\setminus\{w\})$.

Suppose, by way of contradiction, that $y\in N(v)\cap N(w)\cap S_0$.
Since $H[S_0]$ has the minimum number of connected components,   we  conclude that $G_{v}\setminus v$ is connected; otherwise as $G_v$ is a connected component of $H[S_0\cup\{v\}]$, choose a vertex $v'\in V(G_v)$ such that $G_{v}\setminus v'$ remains connected. Define $S'=(S_0\setminus\{v'\})\cup\{v\}$. One can check that
$|S'|=|S_0|$, $H[S']$ is $G$-free, and $H\setminus S'$ contains the same number of $(\Delta(H)-d)$-regular connected subgraphs  as  $H\setminus S_0$. But the number of connected components of $H[S']$ is less
than that of $H[S_0]$, which is impossible.  
As $G_{v}\setminus v$ is connected, there is a shortest path $P$  from  $y$ to $u$ in $G_{v}\setminus v$. Let $y'$ be the last vertex of $P$ in $V(G_{v}\setminus\{v\})\cap V (G_{w}\setminus\{w\})$.
Define $S^*=(S_0\setminus\{y'\})\cup \{w\}$. One can check that $|S^*|=|S_0|$ and $H[S^*]$ is $G$-free. 
 The vertex $y'$ has at least $d+1$ neighbors in $S^*$, because $y'$ has $d$ neighbors in $G_{w}$ and $y'$ is adjacent to  its immediate successor on  $P$ which is not $G_{w}$. Therefore,
   the number of neighbors of $y'$ in $H\setminus S^{*}$ is at most $\Delta(H)-d-1$. 
Thus, $y'$  does not lie in any $(\Delta(H)-d)$-regular subgraph in $H\setminus S^{*}$. This contradicts the assumption that $H\setminus S_0$ has 
the minimum number of  $(\Delta(H)-d)$-regular connected subgraphs.
\end{proof}

Suppose that  $v_0\in V(H_0)$ is not a cut vertex of $H_0$. Choose a vertex  $y_0\in V(G_{v_0})$ such that 
$y_0$ is not a cut vertex in $G_{v_0}$ and $y_0\neq v_0$.
Set $S_1=(S_0\setminus \{y_0\})\cup \{v_0\}$. Since $|S_1|=|S_0|$, $H[S_1]$ is $G$-free, and   $H\setminus S_0$ has as few $(\Delta(H)-d)$-regular connected subgraphs as possible,   
$y_0$ must be in a $(\Delta(H)-d)$-regular subgraph in $H\setminus S_1$, say $H_1$.  Also, 
the number of  components of $H[S_1]$ is equal to that of $H[S_0]$.
 If $V(H_0)\cap V(H_1)\neq \varnothing$, then $V(H_0)\setminus\{v_0\}\subset  V(H_1)$; otherwise there is a vertex in $V(H_0)\cap V(H_1)$
 has degree greater than $\Delta(H)$, which is not possible.  Therefore, $N(v_0)\cap V(H_0)\subseteq V(H_1)$.
  Using the same reasoning as Claim~\ref{cc}, one can show that the induced subgraph by $N(v_0)\cap V(H_1)$ is a complete
  graph and consequely  
  $H_0$ and $H_1$ are isomorphic to the complete graph
  $K_{\Delta(H)-d+1}$. 
 Since $|V(H_0)|\geq 2$, we can choose a  vertex $u$ distinct from $v_0$ in $H_0$ such that 
 $y_0\in N(v_0)\cap N(u)\cap S_0$. Therefore, by using Claim~\ref{c8} we have 
 $G_{v_0}\setminus\{v_0\}= G_{u}\setminus\{u\}$.
 Hence, Claim~\ref{cc} implies the statement.

Suppose that $V(H_0)\cap V(H_1)= \varnothing$.
For $i\geq 1$, 
assume that $H_{i-1}$, $S_{i-1}$, $v_{i-1}$, $G_{v_{i-1}}$, and $y_{i-1}$ are choosen such that
 $v_{i-1}$ is not a cut vertex in $H_{i-1}$, 
$G_{v_{i-1}}$ is a unique copy of $G$ in $H[S_{i-1}\cup\{v_{i-1}\}]$, and
 $y_{i-1}$   is not a cut vertex in $G_{v_{i-1}}$.
Set $S_i=(S_{i-1}\setminus \{y_{i-1}\})\cup \{v_{i-1}\}$.
The vertex $y_{i-1}$ must be in a  $(\Delta(H)-d)$-regular connected subgraph in $H\setminus S_i$, say $H_i$.  Also, 
the number of  components of $H[S_i]$ is equal to that of $H[S_{i-1}]$.
Choose a vertex $v_i\in V(H_i)$ such that  $v_i\neq y_{i-1}$ and  $v_i$ is not a cut vertex of $H_i$.
Assume that $G_{v_{i}}$ is a unique copy of $G$ in $H[S_{i}\cup\{v_{i}\}]$ and choose $y_{i}$  such that $y_i$ is not a cut vertex in $G_{v_{i}}$.

Since $H$ is a finite graph,  there is  the smallest number $\ell$ such that $V(H_\ell)$ intersects $V(H_j)$ for some $j\leq \ell-1$. Without loss of generality assume that $j=0$.
  As the case $V(H_0)\cap V(H_1)\neq \varnothing$, one can show that 
  $V(H_0)\setminus\{v_0\}\subset  V(H_\ell)$.

\begin{claim}\label{cl11} We can assume that
$V(G_{v_0})\setminus\{y_0\}\subseteq S_{\ell}$.
\end{claim}
 \begin{proof}[proof of Claim \ref{cl11}.]
On the contrary, assume that  $ V(G_{v_0})\setminus\{y_0\}\not\subseteq S_{\ell}$.
Let $i$ be the smallest number for which 
 $ V(G_{v_0})\setminus\{y_0\}\not\subseteq S_{i}$. Therefore,
 $y_{i-1}\in V(G_{v_0})\setminus\{y_0\}$.
 Since  $G_{v_0}\setminus\{y_0\}$ is connected in $H[S_{i-1}]$, 
 $G_{v_{i-1}}$ is $d$-regular,
  and $y_{i-1}$ lies in both $G_{v_0}\setminus\{y_0\}$ and
$G_{v_{i-1}}$, we have $V(G_{v_0})\setminus\{y_0\}\subseteq V(G_{v_{i-1}})$.
As $G_{v_0}$ and $G_{v_{i-1}}$ are $d$-regular, it follows that $N(y_0)\cap V(G_{v_0})=N(v_{i-1})\cap V(G_{v_{i-1}}).$

\noindent If $i=2$, then $N(y_0)\cap S_1=N(v_1)\cap S_1$ and consequently $|N(y_0)\cap N(v_1)\cap S_1|=d$.  By using   Claims~\ref{cc}~and~\ref{c8}  we conclude the statement of Part~{\rm (c)}. Therefore, we can assume that $i\geq 3$.

 \noindent Suppose  that    $y_0$ is adjacent to $v_{i-1}$. Since $H_{i-1}$ is $(\Delta(H)-d)$-regular, we have $y_0\in V(H_{i-1})$. Hence,
$y_0\in V(H_1)\cap V(H_{i-1})\neq\varnothing$, which contradicts  the minimality of $\ell$.
Therefore, we can assume that  $y_0$ is not adjacent to $v_{i-1}$. 
Consider the following two cases.
\item[(i)] $d\geq 2$.\\
The vertex $y_0$ has $d$ or $d+1$ neighbors in $S_{i-1}$. 
One of them may be $v_1$ and 
$d$ of them must be in $V(G_{v_0})\setminus\{y_0\}$. 
  We show that $G_{v_{i-1}}$ is  isomorphic to $K_{d+1}$.  
  If two vertices $u,u'$ in $N(v_{i-1})\cap S_{i-1}$ are not adjacent, then
 define $S'=(S_{i-1}\setminus\{u,u'\})\cup\{y_0,v_{i-1}\}$.
  The vertices $y_0$ and $v_{i-1}$ do not lie in any copy $G$ in $H[S']$ because
$y_0$ have at most $d-1$   and $v_{i-1}$ have  $d-2$ neighbors in $S'$, respectively. Thus, $H[S']$ is $G$-free.
 Both vertices $u$ and $u'$ have $d+1$ neighbors in $S'$. Therefore, $u$ and  $u'$ do not lie in any $(\Delta(H)-d)$-regular subgraph in
 $H\setminus S'$,
 which contradicts  $H\setminus S_{i-1}$ contains the minimum possible number  of $(\Delta(H)-d)$-regular connected subgraphs.
 Then $G_{v_{i-1}}$ is isomorphic to $K_{d+1}$.
 
 \noindent For some $u\in N(v_{i-1})\cap S_{i-1}$, define $S''=(S_{i-1}\setminus\{u\})\cup\{y_0,v_{i-1}\}$.
  If $y_0$ has exactly $d$ neighbors in $S_{i-1}$, then each of $y_0$ and $v_{i-1}$ has $d-1$ neighbors in $S''$ and hence $H[S'']$ is $G$-free,  which contradicts    the maximality of $S_{i-1}$. Therefore, we can assume that $y_0$ has exactly $d+1$ neighbors in $S_{i-1}$. Therefore,
  $y_0$ must be adjacent to $v_1$.

\noindent The vertex $v_1$
 is not adjacent to any vertex of $N(v_{i-1})\cap S_{i-1}=N(y_{0})\cap G_{v_0}$; otherwise if $v_1$ has a neighbor in $N(y_{0})\cap G_{v_0}$, then $N(y_{0})\cap N(v_1)\cap S_1\neq \varnothing$. Therefore,
 by Claims~\ref{cc}~and~\ref{c8} we conclude the statement of  Part~{\rm (c)}.
 Since  $v_1$ is not adjacent to any vertex in $N(y_{0})\cap G_{v_0}$ and $y_0$  has $d-1\geq 1$ neighbors in $S''$ which are not adjacent to
 $v_1$, we conclude that $y_0$ cannot lie  in a copy of $G\cong K_{d+1}$ in $H[S'']$. Also, $v_{i-1}$ has $d-1$ neighbors in $S''$.
 Thus,  $H[S'']$ is $G$-free,  which contradicts    the maximality of $S_{i-1}$.
 \item[(ii)] $d=1$ and $|V(H_0)|\geq 3$.\\
 Since $d=1$, we have $G_{v_0}=K_2$
and  $v_0=y_{i-1}$.
 The vertex $v_0$ is adjacent to $y_0$ and has $\Delta(H)-1$ neighbours in $V(H_0)$.
 The vertex $v_{i-1}$ is another neighbour of $v_0(=y_{i-1})$ which is in $H_{i-1}$. Because of $i\geq 3$ we have $v_{i-1}$ is distinct from $y_0\in H_{1}$ and the vertices in $H_0$. Therefore, $\deg(v_0)\geq \Delta(H)+1$ which is not possible.
\end{proof}
   \noindent Since $G_{v_0}\setminus\{y_0\}\subseteq S_{\ell}$, we have $v_0\in S_{\ell}$.
   As the case $V(H_0)\cap V(H_1)\neq \varnothing$, one can show that 
  $V(H_0)\setminus\{v_0\}\subset  V(H_\ell)$ and $H_0$ and $H_{\ell}$ are isomorphic to $K_{\Delta(H)-d+1}$.
  If $|V(H_0)|\geq 3$, choose two vertices $w_1,w_2$ in $V(H_0)\setminus\{v_0\}$. Then,
  $v_0\in N(w_1)\cap N(w_2)\cap S_1$ and hence Claims~\ref{cc}~and~\ref{c8}  imply the statement.
  
Assume that $d\geq 2$ and $w\in V(H_0)\cap V(H_\ell)$. Therefore, $w$ is adjacent to $y_{\ell-1}$ and $v_0$.
 The induced subgraph $H[S_{\ell}\cup\{w\}]$ contains a  unique copy of $G$, say $G_w$. By using Claim~\ref{cl11} and as $G_{v_0}\setminus\{y_0\}$ is connected and $v_0\in V(G_w)\cap (V(G_{v_0})\setminus\{y_0\})$, we have $ (V(G_{v_0})\setminus\{y_0\})\subseteq V(G_w)$.
 Consequently, $N(y_0)\cap V(G_{v_0})=N(w)\cap V(G_{v_0})$. If $w$ is adjacent to $y_0$, 
 then $y_0\in N(v_0)\cap N(w)\cap S_{0}$ and  Claims~\ref{cc}~and~\ref{c8}  imply the statement. Assume that $y_0$ is not adjacent to $w$. The proof of this case  is same as the proof of Claim~\ref{cl11} when $y_0$ is not adjacent to $v_{i-1}$.

\end{proof}
Now we are in the position  to prove  Theorem {\rm\ref{1th}}.
\begin{proof}[\bf{Proof of Theorem~{\rm\ref{1th}}}] Let $\Delta(H)=\sum_{i=1}^{k}d_k$. 
The proof is by induction on $k$. The statement trivially holds for  $k=1$.
Therefore, we may assume that $k\geq 2$.
Let $V_1$ be a  subset of $V(H)$ such that $H[V_1]$ is $G_1$-free and $V_1$ has the maximum possible size.
 Hence, by Lemma \ref{le0}~(a), we have
$\Delta(H\setminus V_1)\leq \Delta(H)-d_1$.
If $H\setminus V_1$ does not contain any  $(\Delta(H)-d_1)$-regular components, 
   then from the induction hypothesis,   $H\setminus V_1$ can be decomposed into $k-1$ subsets $V_2,\ldots,V_k$ such that $H[V_i]$ is $G_i$-free for each $2\leq i\leq k$.
If $H\setminus V_1$  has a  $(\Delta(H)-d_1)$-regular component, then 
by lemma~\ref{le0}~(c), we must have either $G\cong K_{d+1}$  and $H\cong K_{\Delta(H)+1}$,  
 $G\cong K_2$  and $H\cong C_{2\ell+1}$ for some positive integer $\ell$,  
or $H\cong G$, which is not possible.

\end{proof}
 
\bibliographystyle{plain}

\end{document}